\documentclass[12pt,reqno]{amsart}

\usepackage[T5,T1]{fontenc}
\usepackage{amssymb}
\usepackage{hyperref}
\usepackage[all]{xypic} 
\usepackage{bm}
\usepackage{xcolor}
\usepackage{pgffor}
\usepackage{caption}
\usepackage{subcaption}
\usepackage{tikz}
\usepackage{pgffor}

\newcommand{\Gal}{\text{\rm Gal}}

\newcommand{\comment}[1]{}

\makeindex

\begin{document}


\thanks{The second author is partially supported by the Natural Sciences and Engineering Research Council of Canada grant R0370A01.  He also gratefully acknowledges the Faculty of Science Distinguished Research Professorship, Western Science, in years 2004/2005 and 2020/2021. The second and third authors also gratefully acknowledge the support of Western Academy for Advanced Research at Western University.  The fourth author is partially supported by 2021 Wellesley College Faculty Awards. The fifth author is funded by Vingroup Joint Stock Company and supported by Vingroup Innovation Foundation (VinIF) under the project code VINIF.2021.DA00030}

\title[Module structure of elementary $p$-abelian extensions]{Galois module structure of some elementary $p$-abelian extensions}

\author[Lauren Heller]{Lauren Heller}
\address{Department of Mathematics, University of California, Berkeley, 970 Evans Hall \#3840, Berkeley, CA \ 94720-3840 \ USA}
\email{lch@math.berkeley.edu}

\author[J.~Min\'{a}\v{c}]{J\'{a}n Min\'{a}\v{c}}
\address{Department of Mathematics, Western University, London, Ontario, Canada N6A 5B7}
\email{minac@uwo.ca}

\author[T.T.~Nguyen]{Tung T.~Nguyen}
\address{Department of Mathematics, Western University, London, Ontario, Canada N6A 5B7}
\email{tungnt@uchicago.edu}

\author[Andrew Schultz]{Andrew Schultz}
\address{Department of Mathematics, Wellesley College, 106 Central Street, Wellesley, MA \ 02481 \ USA}
\email{andrew.c.schultz@gmail.com}

\author[N.D.~T\^an]{{\fontencoding{T5}\selectfont Nguy\~ \ecircumflex n Duy T\^an}}
\address{School of Applied Mathematics and Informatics, Hanoi University of Science and
Technology, Hanoi, Vietnam}
\email{tan.nguyenduy@hust.edu.vn}

\dedicatory{To Professor Moshe Jarden, with gratitude for his vision and encouragement.}

\date{\today}

\begin{abstract}
We determine the Galois module structure of the parameterizing space of elementary $p$-abelian extensions of a field $K$ when $\Gal(K/F)$ is any finite $p$-group, under the assumption that the maximal pro-$p$ quotient of the absolute Galois group of $F$ is a free, finitely generated pro-$p$ group, and that $F$ contains a primitive $p$th root of unity if $\text{char}(F) \neq p$.
\end{abstract}

\maketitle

\newtheorem*{theorem*}{Theorem}
\newtheorem*{lemma*}{Lemma}
\newtheorem{theorem}{Theorem}
\newtheorem{proposition}{Proposition}[section]
\newtheorem{corollary}[proposition]{Corollary}
\newtheorem{lemma}[proposition]{Lemma}

\theoremstyle{definition}
\newtheorem*{definition*}{Definition}
\newtheorem*{remark*}{Remark}
\newtheorem{example}[proposition]{Example}

\parskip=10pt plus 2pt minus 2pt

\section{Introduction}
Throughout the paper we adopt the following notation: $F$ is a field and $F(p)$ the maximal pro-$p$ extension of $F$.  We write $G_F(p) = \Gal(F(p)/F)$; this is the maximal pro-$p$ quotient of the absolute Galois group of $F$.  We assume throughout that $G_F(p)$ is a free, finitely-generated pro-$p$ group.  We let $K/F$ be an extension so that $G:=\Gal(K/F)$ is a finite $p$-group.  If $\text{char}(F) \neq p$, then we assume $F$ contains a primitive $p$th root of unity $\xi_p$.  We show that these hypotheses are satisfied in many interesting cases that play an important role in Galois theory in Examples \ref{ex:example.char.p} through \ref{ex:example.geometry} below. 

The object we will be investigating is the parameterizing space of elementary $p$-abelian extensions of $K$, which we denote by $J(K)$. In the case where $\text{char}(K) = p$, Artin-Schreier theory tells us that $J(K) = K/\wp(K)$, where $\wp(K) = \{k^p-k:k \in K\}$.   If we have $\text{char}(K)\neq p$ and $\xi_p \in K$, then Kummer theory tells us that $J(K) = K^\times/K^{\times p}$.  In either case, the submodule of elements which have a representative from $F$ plays an important role in the structure of $J(K)$; we denote this module by $[F]$.  More specifically, we have $[F] = (F+\wp(K))/\wp(K)$ when $\text{char}(F) = p$, and $[F] = FK^{\times p}/K^{\times p}$ otherwise.

The most natural structure on $J(K)$ is that of an $\mathbb{F}_p$-vector space. Artin-Schreier and Kummer theory tell us that for any $\ell \in \mathbb{N}$, the $\ell$-dimensional subspaces of $J(K)$ are in bijection with elementary $p$-abelian extensions of $K$ of rank $\ell$.  Since $K$ is a Galois extension of $F$ in the context of this paper, $J(K)$ is more than just a vector space: it is an $\mathbb{F}_p[\Gal(K/F)]$-module.  This  structure encodes additional Galois-theoretic information about elementary $p$-abelian extensions of $K$.  Specifically, the finite $\mathbb{F}_p[\Gal(K/F)]$-submodules of $J(K)$ are in bijection with the elementary $p$-abelian extensions of $K$ which are additionally Galois over $F$.  

In the case where $\Gal(K/F)$ is a cyclic $p$-group, the module structure of $J(K)$ has already been computed.  The surprise in computing the structure of $J(K)$ in that case isn't so much that it has a computable decomposition, but rather that its decomposition is far simpler than one might expect \emph{a priori}. Specifically, it contains at most $2\log_p(|\Gal(K/F)|)+1$ isomorphism classes of indecomposable summands from the $|\Gal(K/F)|$ possibilities. For example, if $\Gal(K/F)$ is a cyclic group of order $p^2$, there are $p^2$ indecomposable cyclic modules with dimensions $1,2,\cdots,p^2$.  In the decomposition of $J(K)$, however, only cyclic modules of dimensions $1$, $2$, $p$, $1+p$, and $p^2$ can occur.  Not only that, but for any given module $M \subseteq J(K)$ and its corresponding elementary $p$-abelian extension $L/K$, there is an explicit dictionary which allows one to connect module- and field-theoretic properties of $M$ to the group structure of $\Gal(L/F)$ (see \cite{MSS1,MS1,Schultz}).  These two ingredients have provided a new approach for analyzing absolute Galois groups, including both automatic realization results and realization multiplicity results which serve to distinguish the class of absolute Galois groups within the larger category of profinite groups  (\cite{BS,CMSHp3,MSSauto,MS2}). 

Outside of the case where $\Gal(K/F)$ is a cyclic $p$-group, there have also been a few explorations into the module structure of $J(K)$ (and related objects) when $\Gal(K/F)$ is noncyclic. The interested reader may consult \cite{CheboluMinac,CMSS,Eimer,MST}.

The feature which distinguishes this work from previous results is the base extension's group structure: in this paper we allow $\Gal(K/F)$ to be an arbitrary finite $p$-group.  With such a mild assumption to work with, the representation theory of $\mathbb{F}_p[\Gal(K/F)]$ is typically intractable to understand holistically.  For example, when $p$ is an odd prime, even $\mathbb{F}_p[\mathbb{Z}/p\mathbb{Z}\oplus\mathbb{Z}/p\mathbb{Z}]$ has a wild representation type.  Despite the complexities from the modular representation viewpoint, we are able to compute a decomposition of $J(K)$ in this setting.  Even more, our decomposition consists of only two isomorphism classes of indecomposable summands: one free, and one not.  The non-free summand is best described in terms of a functorial construction on the class of $\mathbb{F}_p[\Gal(K/F)]$-modules known as the Heller operator, denoted $\Omega_{\mathbb{F}_p[\Gal(K/F)]}$.

We will describe the Heller operator in detail in section \ref{sec:module}. For now, we point out a few salient features of its second iterate on the trivial $\mathbb{F}_p[\Gal(K/F)]$-module $\mathbb{F}_p$ (denoted $\Omega^{2}_{\mathbb{F}_p[\Gal(K/F)]}(\mathbb{F}_p)$) and that module's dual (denoted $\Omega^{-2}_{\mathbb{F}_p[\Gal(K/F)]}(\mathbb{F}_p)$) since the latter plays a key role in our decomposition. In Lemma \ref{le:exact.sequence.for.Heller} we will be able to find $\Omega^{2}_{\mathbb{F}_p[\Gal(K/F)]}(\mathbb{F}_p)$ in a short exact sequence involving $\mathbb{F}_p$, $\mathbb{F}_p[G]$, and a free $\mathbb{F}_p[G]$-module on $d$ generators (where $d$ is the minimum number of generators of $G$). From this short exact sequence, one can determine many structural properties of $\Omega^{\pm 2}_{\mathbb{F}_p[\Gal(K/F)]}(\mathbb{F}_p)$.  For example, again using $d$ to represent the minimal number of generators for $G$, one finds that $$\dim_{\mathbb{F}_p}\left(\Omega^{2}_{\mathbb{F}_p[\Gal(K/F)]}(\mathbb{F}_p)\right)=\dim_{\mathbb{F}_p}\left(\Omega^{-2}_{\mathbb{F}_p[\Gal(K/F)]}(\mathbb{F}_p)\right)=(d-1)|G|+1.$$  With more specific information about the group $G$, one can give even more insight into the structure of these modules.  In section \ref{sec:bicyclic.example}, for example, we study the modules $\Omega^{2}_{\mathbb{F}_p[\Gal(K/F)]}(\mathbb{F}_p)$ and $\Omega^{-2}_{\mathbb{F}_p[\Gal(K/F)]}(\mathbb{F}_p)$ quite carefully in the specific case that $G \simeq \mathbb{Z}/p\mathbb{Z}\oplus\mathbb{Z}/p\mathbb{Z}$, and we see that the former is generated by $3$ elements and has a fixed module of dimension $2$, whereas the latter is generated by $2$ elements and has a fixed module of dimension $3$.

\begin{theorem}\label{th:main.theorem}
Suppose that $F$ is a field so that  the maximal pro-$p$ quotient $G_F(p)$ of the absolute Galois group of $F$ is a free, finitely-generated pro-$p$ group, and suppose that $K/F$ is an extension so that $\Gal(K/F)$ is a finite $p$-group. If $\text{char}(K) \neq p$, assume that $\xi_p \in K$.  Let $J(K)$ be the parameterizing space of elementary $p$-abelian extensions of $K$, so that $J(K)=K^\times/K^{\times p}$ if $\text{char}(K) \neq p$ and $J(K) = K/\wp(K)$ otherwise.  Then as an $\mathbb{F}_p[\Gal(K/F)]$-module, $$J(K) \simeq \Omega_{\mathbb{F}_p[\Gal(K/F)]}^{-2}(\mathbb{F}_p) \oplus Y,$$  where $Y$ is a free module of rank $\dim [F]$.
\end{theorem}

\begin{example}\label{ex:example.char.p}
The hypotheses of Theorem \ref{th:main.theorem} hold for any field $F$ of characteristic $p$ for which $ \dim F/\wp(F) < \infty$. This is a consequence of a theorem of Witt which tells us that $G_F(p)$ is a free pro-$p$ group for all fields $F$ with $\text{char}(F) = p$. 
\end{example}

\begin{example}
A field $L$ is said to be Pythagorean if each finite sum of squares in $L$ is again a square in $L$, and a Pythagorean field is called formally real if $-1 \not\in L^{\times 2}$.  A formally real Pythagorean field $L$ for which $|L^\times/L^{\times 2}| = 2^{n}$ is said to have the strong approximation property (SAP) if $L$ admits exactly $n$ orderings.   Suppose that $\hat F$ is an SAP field, and let $F = \hat F(\sqrt{-1})$.  Then $G_F(2)$ is a free pro-$2$ group of rank $n-1$; this is proved in \cite[Le.~4.6]{MRT} using a group theoretic approach, and in \cite[Prop.~3.4.15]{Rogelstad} using a combination of Galois cohomology and field theory.  Hence the results of Theorem \ref{th:main.theorem} hold for such a field.  (For more information on SAP fields, the interested reader can consult sections 4.1 and 4.2 from \cite{MRT}, as well as \cite{Mi} and chapter 17 of \cite{Lam}.)

Notice in particular that we can realize an abundance of fields of this type by using a combination of \cite[Prop.~1.3]{EfratHaran} and \cite[Th.~4.1]{MRT}: the former tells us that if $n \in \mathbb{N}$ is given and $K_1,\cdots,K_n$ are fields of characteristic $0$ for which $G_{K_i}(2) \simeq \mathbb{Z}/2\mathbb{Z}$, then there exists a field $\hat F$ with $\text{char}(\hat F)=0$ whose absolute Galois group is the free product of $n$ copies of $\mathbb{Z}/2\mathbb{Z}$ (in the category of pro-$2$ groups); then the latter result tells us that this field is an SAP field.  
\end{example}

\begin{example}\label{ex:example.PAC}
A field $L$ is called pseudo algebraically closed (PAC) if it has the property that every nonempty variety defined over $L$ has an $L$-rational point.  It is known that the absolute Galois group of such a field is projective (see \cite[Th.~11.6.2]{FJ}), and hence its $p$-Sylow subgroup is free.  Moreover, from \cite[Th.~23.1.2]{FJ} we have that any projective group can be realized as the absolute Galois group of a perfect PAC field. Since projective groups and free pro-$p$ groups are the same in the category of pro-$p$ groups (see \cite[Th.~4.8]{Koch} or \cite[Cor.~22.7.6]{FJ}), these fields give a bounty of examples in which the hypotheses of Theorem \ref{th:main.theorem} are applicable.
\end{example}

\begin{example}\label{ex:example.geometry}
There are a number of fields that arise from geometric considerations that have free absolute Galois groups.  These results are connected to Shafarevich's conjecture that the absolute Galois group of $\mathbb{Q}^{\text{ab}}$ is free; this conjecture has been generalized to state that the absolute Galois group of the maximal cyclotomic extension of any global field $K$ is free.  The function field case was settled in \cite{Harbater} and \cite{Pop}, and generalizations of Shafarevich's conjecture have been explored in the function field case in \cite{HarbaterFunctionFields} as well.  (See also \cite{Pop2}, and \cite{HarbaterExposition} for a nice exposition of the Shafarevich Conjecture.)  Hence Theorem \ref{th:main.theorem} can be applied to finitely generated subgroups of the $p$-Sylow subgroup of these absolute Galois groups. 

It is worth observing that the freeness of the $p$-Sylow subgroup of $\Gal(\mathbb{Q}^{\textrm{ab}})$ can be proved unconditionally without appealing to Shafarevich's conjecture.  In fact, the proof is quite simple using the local-global principle and the fact that a pro-$p$-group $G$ for which the second cohomology $H^2(G,\mathbb{F}_p)$ vanishes is necessarily a free pro-$p$-group.  (See \cite[Th.~4.12]{Koch}.) The interested reader can consult Observation 3.1 from \cite{BSJN} for further details of this argument.
\end{example}

The paper proceeds as follows. In section \ref{sec:prelims} we review the basic module-, group-, and field-theoretic information we will need to proceed with our argument.  The proof of the main theorem comes in section \ref{sec:proof.of.main.theorem}.  In section \ref{sec:computing.Galois.groups} we begin the investigation of how to compute $\Gal(L/F)$ when $L$ is an extension corresponding to a submodule $N \subseteq J(K)$ by settling this question in the specific case where $N$ is one of the summands from our decomposition.  Finally, in section \ref{sec:bicyclic.example}, we compute the explicit module structure of the non-free summand in the case where $\Gal(K/F) \simeq \mathbb{Z}/p\mathbb{Z}\oplus\mathbb{Z}/p\mathbb{Z}$.

\subsection*{Acknowledgements}

We are very grateful to our collaborators on previous projects associated to this investigation, including S.~Chebolu, F.~Chemotti, P.~Guillot, J.~Swallow, and A.~Topaz. We are also very grateful for productive discussions with A.~Eimer concerning related research. Finally, we thank an anonymous referee for their help in improving the quality and clarity of this paper.

\section{Preliminaries}\label{sec:prelims}  

\subsection{Module-theoretic properties}\label{sec:module}

Our main result concerns the decomposition of an $\mathbb{F}_p[G]$-module when $G$ is a finite $p$-group, so we will need a few module-theoretic results. When we describe properties of a generic $\mathbb{F}_p[G]$-module $M$ in this section, we will assume that $M$ is an additive module, and hence the $G$-action will be written multiplicatively.  In particular, the trivial element will be denoted $0_M$.

First we record a lemma that gives us a convenient way to detect when two submodules are independent.

\begin{lemma}\label{le:exclusion.lemma}
Suppose that $G$ is a finite $p$-group and that $M$ is an $\mathbb{F}_p[G]$-module containing submodules $U$ and $V$.  Then $U \cap V = \{0_M\}$ if and only if $U^G \cap V^G = \{0_M\}$.
\end{lemma}

\begin{proof}
The fact that $U \cap V = \{0_M\}$ implies $U^G \cap V^G = \{0_M\}$ is obvious.  For the other direction, recall from \cite[Lem.~4.13]{Koch} that the action of every finite $p$-group on a nonzero $\mathbb{F}_p$-vector space has a nonzero fixed vector.  Hence if $U \cap V \neq \{0_M\}$, then there is some nontrivial element in $U^G \cap V^G$.  
\end{proof}

The following basic fact from representation theory will give us some additional useful consequences.

\begin{lemma}\label{eq:fixed.part.of.free.module}
For a finite group $G$ and any ring $R$, the fixed submodule of $R[G]$ is $\langle \sum_{g \in G} g \rangle_{R}$.
\end{lemma}

\begin{proof}
Clearly any multiple of $\sum_{g \in G} g$ is fixed by $G$, so let $\gamma$ be a fixed element of $R[G]$, and let $c_g \in R$ be given so that $\gamma= \sum_{g \in G} c_g g$.  Now observe that for any $\hat g \in G$ we have $\hat g \gamma = \gamma$.  Comparing the coefficients of the identity element $e_G$ on both sides of this equation, we find that $c_{\hat g^{-1}} = c_{e_G}$.  Since this is true for all $\hat g \in G$, we have that $\gamma = c_{e_G}\left(\sum_{g \in G} g\right)$ as desired.
\end{proof}

\begin{corollary}[Cf.~{\cite[Lem.~2.4]{AMT}}]\label{cor:all.ideals.contain.norm}
Suppose that $G$ is a finite $p$-group.  Any nontrivial left ideal of $\mathbb{F}_p[G]$ contains the element $\sum_{g \in G} g$.
\end{corollary}

\begin{proof}
Following the proof of Lemma \ref{le:exclusion.lemma}, any nontrivial left ideal $I$ contains a nontrivial fixed element.  From Lemmas \ref{le:exclusion.lemma} and \ref{eq:fixed.part.of.free.module}, we have the desired result.  
\end{proof}

%
%

To conclude this subsection, we consider a module construction that arises from projective resolutions.  Let $M$ be a finitely-generated $\mathbb{F}_p[G]$-module, and let $P$ be a minimal projective cover of $M$ (in the category of $\mathbb{F}_p[G]$-modules), meaning a surjection $\varphi:P \twoheadrightarrow M$ so that $P$ is a projective $\mathbb{F}_p[G]$-module and no summand of $P$ lies in the kernel of $\varphi$.  Observe that a minimal projective cover exists and that it is unique up to isomorphism (see \cite[Th.~4.1]{Carlson} and its proof). Not surprisingly, our  minimal projective cover $P$ is a projective module in the category of $\mathbb{F}_p[G]$-modules of smallest dimension such that there exists a surjective homomorphism from this module onto $M$.  The proof of uniqueness is a simple but very nice application of Fitting's lemma.  (Details are in \cite[pg.~13]{CarlsonBook}.) 

We then define the Heller shift of $M$ to be $\Omega(M) := \ker(\varphi)$.   Iterating this procedure, we set $\Omega^1(M) = \Omega(M)$ and $\Omega^n(M):=\Omega\left(\Omega^{n-1}(M)\right)$ for $n\geq 2$.   This construction can be dualized as well, by defining $\Omega^{-d}(M) := \left(\Omega^{d}(M^*)\right)^*$.  We can also define $\Omega^0(M) = \Omega^{-1}(\Omega^1(M))$; it follows that there exists some projective $\mathbb{F}_p[G]$-module $Q$ so that $M \simeq \Omega^0(M) \oplus Q$.

These modules satisfy a number of useful properties that are summarized nicely in \cite[Prop.~4.4]{CarlsonBook}. For example, suppose that $n$ and $m$ are integers, and that $N$ and $M$ are $\mathbb{F}_p[G]$-modules.  Then we have $\Omega^{n+m}(M) \simeq \Omega^n\left(\Omega^m(M)\right)$ and $\Omega^n(M\oplus N) \simeq \Omega^n(M) \oplus \Omega^n(N)$.  We also have that $\Omega^n(M)$ has no nonzero projective submodules, and $\Omega^n(M) = 0$ if and only if $M$ is projective.
%
%
%

In what follows, we will be particularly interested in $\Omega^{2}(\mathbb{F}_p)$ when $G$ is a finite $p$-group, and so it will be useful to have a more concrete description of this module.

\begin{lemma}[{\cite[Eq.~(2.16)]{Carlson}}]\label{le:exact.sequence.for.Heller}
Let $G$ be a finite $p$-group, and write $\sigma_1,\cdots,\sigma_d$ for a minimal set of generators.  Let $P_1$ be the free $\mathbb{F}_p[G]$-module on generators $c_1,\cdots,c_d$.  Let $\varepsilon:\mathbb{F}_p[G] \to \mathbb{F}_p$ be the augmentation map defined by $\varepsilon(g) = 1$ for all $g \in G$, and let $\partial:P_1 \to \mathbb{F}_p[G]$ be given by $\partial(c_i) = \sigma_i-1$.  Then $\ker(\partial) = \Omega^2(\mathbb{F}_p)$, and the following sequence is exact: $$\xymatrix{0\ar[r] &\Omega^2(\mathbb{F}_p) \ar[r] &\displaystyle P_1 \ar[r]^-\partial & \mathbb{F}_p[G] \ar[r]^-\varepsilon & \mathbb{F}_p \ar[r] &0}.$$
\end{lemma}

Given our interest in module decompositions, it will be useful for us to know that $\Omega^2(\mathbb{F}_p)$ and $\Omega^{-2}(\mathbb{F}_p)$ are indecomposable.  Since $\mathbb{F}_p$ is indecomposable, this follows from

\begin{lemma}\label{le:Omega2.is.indecomposable}
Let $G$ be a finite $p$-group, and $M$ a finitely-generated $\mathbb{F}_p[G]$-module.  If $M$ is indecomposable, then $\Omega^n(M)$ is indecomposable for all $n \in \mathbb{Z}$.
\end{lemma}

\begin{proof}
Since $\Omega^n(M)$ is defined in terms of iterated applications of $\Omega^1$ or $\Omega^{-1}$, it suffices to prove the result for $\Omega^1(M)$ and $\Omega^{-1}(M)$.  We will focus on the former, though the proof of the latter is analogous.

Suppose then that $\Omega(M) = A\oplus B$.  Applying $\Omega^{-1}$ we find $\Omega^0(M) \simeq \Omega^{-1}(A)\oplus \Omega^{-1}(B)$.  But then there is some projective $Q$ with $$M \simeq \Omega^0(M) \oplus Q \simeq \Omega^{-1}(A)\oplus \Omega^{-1}(B) \oplus Q.$$  Since $M$ is indecomposable we know that two of these summands are trivial.  If $Q \neq 0$ then we must have both $\Omega^{-1}(A)=0$ and $\Omega^{-1}(B)=0$, which implies that $A$ and $B$ are projective.  But $\Omega(M)$ cannot have any nontrivial projective submodules, and so we conclude that $A=B=0$ in this case.  This gives $\Omega(M) = 0$, which is indeed indecomposable.  Otherwise we have $Q = 0$, and one of $\Omega^{-1}(A)$ or $\Omega^{-1}(B)$ is trivial; without loss, assume the former.  Again we conclude that $A$ must be projective, which again implies $A=0$.  Hence our decomposition becomes $\Omega(M) = 0 + B$.
%
\end{proof}

\subsection{The Frattini module}\label{sec:frattini}

We provide a survey of some results connected to Frattini covers of profinite groups.  The results from this section can readily be found in more detail in \cite{FJ}.  

The Frattini subgroup of $G$ is $\Phi(G) = \bigcap H$, where the intersection is taken over all maximal, proper, closed subgroups $H$ of $G$.  In the case where $G$ is a pro-$p$ group, one has $\Phi(G) = G^p[G,G]$, the topological subgroup of $H$ generated by $p$th powers and commutators.  Hence in this context $\Phi(G)$ can be thought of as the minimal subgroup $N$ of $G$ for which $G/N$ is elementary $p$-abelian.  

A Frattini cover of a group $G$ consists of a group $X$ and a surjection $\theta:X \twoheadrightarrow G$ so that $\ker(\theta) \subseteq \Phi(G)$.  We can then impose an order on Frattini covers as follows: if $\theta_1:X_1 {\twoheadrightarrow} G$ and $\theta_2:X_2 {\twoheadrightarrow} G$ are two Frattini covers, then $\theta_1 \geq \theta_2$ if there exists some $\psi: X_1 \to X_2$ so that 
$$\xymatrix{
X_1 \ar[rr]^-{\psi} \ar@{->>}[rd]_-{\theta_1} && X_2 \ar@{->>}[ld]^-{\theta_2}\\
&G&}
$$ commutes.  (One can in fact argue that the map $\psi$ is therefore surjective.) With this ordering, one can show that up to isomorphism there is a unique maximal Frattini cover $\widetilde{G}$ of $G$, which is called the universal Frattini cover.  Our interest in universal Frattini covers is in the context  of finite $p$-groups.  In the case that $G$ is a finite $p$-group, one can show that the universal Frattini cover is the free pro-$p$ group on $d(G)$ generators (\cite[Cor.~22.7.8]{FJ}), where here $d(G)$ denotes the minimal number of generators of $G$. 

More specifically, if $G$ is a finite $p$-group and $\theta: \widetilde{G} \twoheadrightarrow G$ its universal Frattini cover, we will later be interested in a particular $\mathbb{F}_p[G]$-module related to $\widetilde{G}$, namely $M_0:=\ker(\theta)/\Phi(\ker(\theta))$.  (The action of $G$ on $M_0$ comes by conjugation.)  The module $M_0$ is the first in a family of modules known as the Frattini modules, which are built up along the tower $\widetilde{G} \twoheadrightarrow G$.  Fortunately, the structure of $M_0$ (and, indeed, the other Frattini modules) has been computed in terms of the Heller operator on the trivial $\mathbb{F}_p[G]$-module $\mathbb{F}_p$.

\begin{proposition}[{\cite{Gaschutz} or \cite[Lem.~2.3]{Fried95}}]\label{prop:Gaschutz}
As an $\mathbb{F}_p[G]$-module, we have $M_0 \simeq \Omega^2(\mathbb{F}_p)$.
\end{proposition}


\subsection{Field-theoretic properties}\label{sec:field.stuff}

Recall that $J(K)$ is the notation we choose for the parameterizing space of elementary $p$-abelian extensions of $K$. When discussing elements of the module $J(K)$, we will write $[k]$ for the class of $J(K)$ represented by an element $k \in K$.  If we need to consider elements from $J(L)$ for some other field $L$, we will express such elements in the form $[l]_L$.  We will adopt the convention of writing $J(K)$ with an additive $\mathbb{F}_p[\Gal(K/F)]$-structure, even in those cases where $J(K)$ is multiplicative.  Not only does this allow for a more uniform presentation of the results, but it avoids cumbersome exponentiation that is difficult to read.  The trivial element in $J(K)$ will therefore be denoted $[0]$.

In the case where $\text{char}(K) = p$, Artin-Schreier theory tells us that $J(K) = K/\wp(K)$, where $\wp(K) = \{k^p-k: k \in K\}$.  If we let $\theta_a$ denote a root of $x^p-x-a$ for $a \in K$, then we have a correspondence between (finite) elementary $p$-abelian extensions of $K$ and (finite) $\mathbb{F}_p$-subspaces of $J(K)$:
\begin{align*}
N \subseteq J(K) &\longmapsto K(\theta_n: [n] \in N)/K\\
L/K&\longmapsto \{[k] \in J(K): k \in K \cap\wp(L)\}.
\end{align*}
When $\text{char}(K) \neq p$ and $K$ contains a primitive $p$th root of unity $\xi_p$,  Kummer theory tells us that $J(K) = K^\times/K^{\times p}$.  We have a correspondence between (finite) elementary $p$-abelian extensions of $K$ and (finite) $\mathbb{F}_p$-subspaces of $J(K)$:
\begin{align*}
N \subseteq J(K) &\longmapsto K(\root{p}\of{n}: [n] \in N)/K\\
L/K&\longmapsto \{[k] \in J(K): k \in K^\times \cap L^{\times p}\}.
\end{align*}
In either case, if $L$ is such an extension and $N$ is its corresponding subspace, then it is an easy exercise to show that $L/F$ is Galois if and only if $N$ is an $\mathbb{F}_p[G]$-submodule of $J(K)$. There is a natural perfect pairing $$\Gal(L/K) \times N \to \mathbb{F}_p$$ (though if $\text{char}(K) \neq p$ this requires an identification of $\mathbb{F}_p$ with $\langle \xi_p \rangle$).  In the Artin-Schreier case, this is given by $\langle \tau,[n]\rangle = \tau(\theta_n)-\theta_n$, whereas in the Kummer theoretic case it is given by $\langle \tau,[n]\rangle = \frac{\tau(\root{p}\of{n})}{\root{p}\of{n}}$.  Furthermore, the natural action of $\Gal(K/F)$ on the components of this perfect pairing is respected: for any $\sigma \in \Gal(K/F)$, $\tau \in \Gal(L/K)$ and $[n] \in N$ we have
$$\langle \sigma \tau \sigma^{-1},[n]^{\sigma}\rangle = \sigma\left(\langle \tau,[n]\rangle\right).$$  This means that $\Gal(L/K)$ and $N$ are dual to each other as $\mathbb{F}_p[\Gal(K/F)]$-modules.

We have already seen that for $\mathbb{F}_p[G]$-modules, the element $\sum_{g \in G} g$ plays an important role.  Observe that when $G = \Gal(L/E)$ with $\text{char}(E) = p$, then the action of of $\sum_{g \in G} g$ on elements of $L$ is the trace map $\text{Tr}_{L/E}:L \to E$.  Naturally this descends to a map $\text{Tr}_{L/E}: J(L) \to J(E)$, though for our purposes it will often be useful to consider the composition \begin{equation}\label{eq:trace.on.classes}\xymatrix{J(L)\ar[r]^-{\text{Tr}_{L/E}} & J(E) \ar[r]^-{\iota} & \frac{E+\wp(L)}{\wp(L)}=[E]_L},\end{equation} where $\iota$ is the natural inclusion.  In the same way, when $\text{char}(E) \neq p$ and $\xi_p \in E$, we will often be interested in the composition \begin{equation}\label{eq:norm.on.classes}\xymatrix{J(L)\ar[r]^-{\text{N}_{L/E}} & J(E) \ar[r]^-{\iota} & \frac{E^\times L^{\times p}}{L^{\times p}} = [E]_L}.\end{equation}  

\begin{lemma}\label{le:norm.trace.is.surjective}
Suppose that $G_E(p)$ is a free pro-$p$ group, and that either $\text{char}(E) = p$ or that $\text{char}(E)\neq p$ and $\xi_p \in E$. Let $\Gal(L/E)$ be a finite $p$-group.  If $\text{char}(E) = p$, then the trace map $\text{Tr}_{L/E}: L \to E$ is surjective.  If instead $\xi_p \in E$, then the norm map $N_{L/E}:L \to E$ is surjective.  

In particular, in either case we have that the action of $\sum_{\sigma \in \Gal(L/E)} \sigma$ which maps $J(L)$ to $[E]_L$ according to equation (\ref{eq:trace.on.classes}) or (\ref{eq:norm.on.classes}) is surjective.
\end{lemma}

\begin{proof}
Let $G=\Gal(L/E)$.  Our proof will be broken into cases depending on whether $\text{char}(E) = p$ or not.  In either case, though, we first show the claim for degree $p$ extensions.  With this in hand, the full result follows since we can view $L/E$ as a tower of such subextensions whose trace/norm map factors along the trace/norm maps from this tower.  To see this, recall that every $p$-group has a nontrivial center, and hence we may select an element $a_1 \in Z(G)$ of order $p$.  Galois theory then says that $\langle a_1 \rangle \lhd G$ corresponds to a subextension $L/\tilde{E}$ of degree $p$ so that $\Gal(\tilde{E}/E)$ is a $p$-group; one proceeds inductively from here.  It will also be useful to know that since $G_E(p)$ is a free pro-$p$ group, so too is $G_{E'}(p)$ for each field $E'$ within the tower over $L$. This follows because $G_{E'}(p)$ is a closed subgroup of $G_E(p)$, and hence is a free pro-$p$ group (see \cite[Th.~4.12, Th.~5.3]{Koch}).  So for the duration of the proof we simply assume that $L/E$ has $\Gal(L/E) \simeq \mathbb{Z}/p\mathbb{Z}$ and that $G_E(p)$ is a free pro-$p$ group.

First, we handle the case where $\text{char}(E) = p$.  Let $L = E(\theta_a)$, where $\theta_a$ is a root of $x^p-x-a$.  Write $\sigma$ for the generator of $\Gal(L/E)$ which satisfies $\sigma(\theta_a)=\theta_a+1$.  For a given $e \in E$, we now show that $\text{Tr}_{L/E}(-e\theta_a^{p-1}) = e$.  To do this, write $x+1 = \theta_a^{p-1}$, and use the fact that $\theta_a^p-\theta_a-a=0$ to conclude that $x = \frac{a}{\theta_a}$.  We therefore have $$x^p+x^{p-1}-a^{p-1} = \frac{a^p}{\theta_a^p}+\frac{a^{p-1}}{\theta_a^{p-1}}-a^{p-1}  =\frac{a^{p-1}}{\theta_a^p}\left(a+\theta_a-\theta_a^p\right) =0.$$ Examining the coefficient of $x^{p-1}$, we conclude $\text{Tr}_{L/E}(x)=-1$, and so $\text{Tr}_{L/E}(-e\theta_a^{p-1}) = -e\text{Tr}_{L/E}(x+1) = e$. (One can prove this result using Newton's identities; see \cite{MinacNewton}.)

Now suppose that $E$ satisfies $\text{char}(E) \neq p$ and $\xi_p \in E$.  Let $e \in E$ be given.  By Kummer theory this means there exists some $\hat e \in E$ with $L = E(\root{p}\of{\hat e})$.  Now since $G_E(p)$ is free, by \cite[Th.~4.12]{Koch} we have that $H^2(G_{E}(p),\mathbb{F}_p) = \{0\}$, so that $(e) \cup (\hat e)$ vanishes.  But it is well known (see, e.g., \cite[Lemma~8.4]{Srin}) that this vanishing is equivalent to $e \in N_{E(\root{p}\of{\hat e})/E}(E(\root{p}\of{\hat e}))=N_{L/E}(L)$, which is the desired result.  
\end{proof}

The previous result is a key ingredient in the following:

\begin{lemma}\label{le:elements.of.f.are.fixed.parts.of.free}
Let $E$ be a field so that $G_E(p)$ is a free pro-$p$ group, and so that either $\text{char}(E) = p$ or $\text{char}(E) \neq p$ and $\xi_p \in E$.  Let $L$ be a finite extension of $E$ so that $\Gal(L/E)$ is a $p$-group.  For any $e \in E$ for which $[e]_L \neq [0]_L$, there exists some $\ell_e \in L$ so that $\langle [\ell_e]_L \rangle^{\Gal(L/E)} = \langle [e]_L \rangle$, and so that $\langle [\ell_e]_L\rangle \simeq \mathbb{F}_p[\Gal(L/E)]$.
\end{lemma}

\begin{proof}
By Lemma \ref{le:norm.trace.is.surjective}, let $[\ell_e]_L$ be given so that $\sum_{\sigma \in \Gal(L/E)} \sigma [\ell_e]_L = [e]_L$.  We claim now that $\langle [\ell_e]\rangle \simeq \mathbb{F}_p[\Gal(L/E)]$.  For this, note that any nontrivial ideal of $\mathbb{F}_p[\Gal(L/E)]$ must contain the element $\sum_{\sigma \in \Gal(L/E)} \sigma$ by Corollary \ref{cor:all.ideals.contain.norm}.  But since the image of $[\ell_e]$ under $\sum_{\sigma \in \Gal(L/E)} \sigma$ is the nontrivial element  $[e]_L$, it follows then that $\text{ann}_{\mathbb{F}_p[\Gal(L/E)]}([\ell_e]) = \{0\}$.  Hence $\langle [\ell_e] \rangle \simeq \mathbb{F}_p[\Gal(L/E)]$, and so $\langle [\ell_e] \rangle^{\Gal(L/E)} = \langle \sum_{\sigma \in \Gal(L/E)}\sigma [\ell_e] \rangle = \langle [e]_L\rangle$.
\end{proof}

\section{Proof of Theorem \ref{th:main.theorem}}\label{sec:proof.of.main.theorem}

In this section we prove the main result.  Recall that $K/F$ is an extension of fields with $\Gal(K/F)$ a finite $p$-group and so that $G_F(p)$ is a free, finitely-generated pro-$p$ group.  For the sake of simplicity, we will write $G$ in place of $\Gal(K/F)$ in this section.  We will let $n \in \mathbb{N}$ be given so that $p^n:=|J(F)|$.

We know that $G/\Phi(G)$ is the maximal elementary $p$-abelian quotient of $G$.  From the perspective of field theory, this means that $\text{Fix}(\Phi(G))$ is an extension of $F$ whose Galois group is elementary $p$-abelian; we can write $p^d = [\text{Fix}(\Phi(G)):F]$.  Hence there exists some $d$-dimensional subspace $\langle [a_1]_F,\cdots,[a_d]_F\rangle \subseteq J(F)$ to which $\text{Fix}(\Phi(G))$ corresponds.  Recall here that $d$ is equal to $d(G)$ --- the minimal number of generators for $G$.

\begin{lemma}\label{le:free.pro.p.part.above.G}
Let $\langle [a_1]_F,\cdots,[a_d]_F\rangle \subseteq J(F)$ correspond to the maximal elementary $p$-abelian extension of $F$ within $K$.  Then there exists an extension $\Lambda/K$ so that $\Gal(\Lambda/F)$ is the free pro-$p$ group on $d$ generators.  In particular the natural surjection $\theta:\Gal(\Lambda/F) \twoheadrightarrow \Gal(K/F)$ from Galois theory is the universal $p$-Frattini extension, and the tower of fields $\Lambda/K/F$ gives a solution to this embedding problem.
\end{lemma}

\begin{proof}
We may complete $\{[a_1]_F,\cdots,[a_d]_F\}$ to an $\mathbb{F}_p$-basis $\{[a_1]_F,\cdots,[a_d]_F,[a_{d+1}]_F,\cdots,[a_n]_F\}$ of $J(F)$, and we let $\sigma_1,\cdots,\sigma_n \in G_F(p)$ be dual to these elements.  Note that $\sigma_1,\cdots,\sigma_n$ form a minimal generating set of $G_F(p)$.  Now define $H=\langle \sigma_1,\cdots,\sigma_d\rangle$, the closed subgroup generated by those dual to $\{[a_1]_F,\cdots,[a_d]_F\}$.  Then we see that $H$ is a free pro-$p$ group as well by showing that for any pro-$p$ group $T$ and any elements $t_1,\cdots,t_d \in T$, there exists a morphism of pro-$p$ groups which takes $\sigma_i$ to $t_i$ for all $1 \leq i \leq d$ (see, e.g., \cite[Th.~4.6]{Koch}). To see this, note that since $G_F(p)$ is a free pro-$p$ group on the generators $\{\sigma_i\}_{i=1}^n$ we can define a map $\varphi:G_F(p) \to T$ by setting $\varphi(\sigma_i) = t_i$ for all $1 \leq i \leq d$, and setting $\varphi(\sigma_j)=\text{id}_T$ for all $d+1 \leq j \leq n$.  But then the restriction $\varphi|_H:H \to T$ provides the desired map.

Note, however, that we then have that $H$ is isomorphic to a quotient of $G_F(p)$ as well --- one simply defines $\psi:G_F(p) \to H$ by setting $$\psi(\sigma_i) = \left\{\begin{array}{ll}\bar \sigma_i,&1 \leq i \leq d\\1,&d+1 \leq i \leq n.\end{array}\right.$$  We then let $\Lambda = \text{Fix}(\ker(\psi))$.

The fact that $\theta:\Gal(\Lambda/F) \twoheadrightarrow \Gal(K/F)$ is the universal $p$-Frattini extension follows from the fact that $\Gal(K/F)$ is a finite $p$-group of rank $d$, and $\Gal(\Lambda/F)$ is the free pro-$p$ group on $d$ generators.  Galois theory tells us  $\Lambda/K/F$ provides a solution to the embedding problem.
\end{proof}

Let $L$ be the maximal $p$-abelian extension of $K$ contained in $\Lambda$; in other words, let $L = \text{Fix}(\Phi(\Lambda/K))$.  Let $X$ be the associated $\mathbb{F}_p$-space in $J(K)$.  Observe that the maximality of $L$ means that $X$ must be closed under the action of $G$, and hence $X$ must be an $\mathbb{F}_p[G]$-submodule; as a consequence we have that $L/F$ is Galois.  

Now we know that $\Gal(\Lambda/K)$ is the kernel of $\theta:\Gal(\Lambda/F) \twoheadrightarrow \Gal(K/F)$ by Galois theory, and so $\Phi(\ker(\theta)) = \Phi(\Gal(\Lambda/K))$.  We know that the Frattini subgroup of $\Gal(\Lambda/K)$ corresponds to the maximal elementary $p$-abelian extension of $K$ within $\Lambda$, however, and so we have $\Phi(\ker(\theta)) = \Gal(\Lambda/L)$.  Hence --- in the notation of section \ref{sec:module} --- we have $M_0 = \Gal(\Lambda/K)/\Gal(\Lambda/L) \simeq \Gal(L/K)$, and so Proposition \ref{prop:Gaschutz} gives $\Gal(L/K) \simeq \Omega^2(\mathbb{F}_p)$.  Since $\Gal(L/K)$ is dual to $X$, we therefore have $X \simeq \Omega^{-2}(\mathbb{F}_p)$.


Recall that $[F]$ denotes the classes from $J(K)$ represented by elements of $F$, and let $\mathcal{I}$ be a basis for $[F]$.  For any given $[f] \in \mathcal{I}$, Lemma \ref{le:elements.of.f.are.fixed.parts.of.free} gives an element $k_f \in K$ so that $\langle [k_f] \rangle^{\Gal(K/F)} \simeq \langle [f] \rangle$ and $\langle [k_f] \rangle \simeq \mathbb{F}_p[G]$. We define $Y = \sum_{[f] \in \mathcal{I}} \langle [k_f]\rangle$.  Lemma \ref{le:exclusion.lemma} tells us that in fact $Y = \bigoplus_{i \in \mathcal{I}} \langle [k_f]\rangle$, and Lemma \ref{le:elements.of.f.are.fixed.parts.of.free} tells us that $Y$ is free.

To see that $X + Y = X \oplus Y$, we argue that $X \cap [F]$ is trivial; if we can show this, then $$Y^G = \left(\bigoplus_{i \in \mathcal{I}} \langle k_f \rangle\right)^G = \bigoplus_{i\in \mathcal{I}} \langle [k_f]\rangle^G = \bigoplus_{i \in \mathcal{I}} \langle [f] \rangle = [F]$$ implies $X^G \cap Y^G$ is trivial, whence $X+Y = X\oplus Y$ by Lemma \ref{le:exclusion.lemma}.  Now to see that $X \cap [F]$ is trivial, note that if there exists $f \in F$ so that $[f] \in X$, then the extension of $F$ corresponding to $\langle [a_1]_F,\cdots,[a_d]_F,[f]_F\rangle$ would be contained in $\Lambda$.  Moreover we must have $[f]_F$ independent from $\{[a_1]_F,\cdots,[a_d]_F\}$ since $[f] \neq 0$ whereas $[a_i]=0$ for all $i$.  Hence $\Gal(\Lambda/F)$ contains a quotient isomorphic to $(\mathbb{Z}/p\mathbb{Z})^{\oplus d+1}$.  On the other hand, if we let $F'$ be the extension of $F$ corresponding to $\langle [a_1]_F,\cdots,[a_d]_F\rangle$ then the Galois-theoretic surjection $\Gal(\Lambda/F) \twoheadrightarrow \Gal(F'/F)$ must be the universal $p$-Frattini cover since $\Gal(F'/F)$ is a $p$-group on $d$ generators, and $\Gal(\Lambda/F)$ is a free pro-$p$ group on $d$ generators by construction.  Hence the maximal elementary $p$-abelian quotient of $\Gal(\Lambda/F)$ is rank $d$, contrary to above.  

Finally, we show that $X\oplus Y = J(K)$ by a dimension count.  Since $Y$ is free with $Y^G = [F]$ and $\dim[F] = n-d$, we see $\dim(Y) = |G|(n-d)$.  The exact sequence from Lemma \ref{le:exact.sequence.for.Heller} gives us $\dim(\Omega^2_{\mathbb{F}_p[G]}(\mathbb{F}_p)) = d|G|-|G|+1 = (d-1)|G|+1$ . So we have $\dim(X\oplus Y) = |G|(n-1)+1$. On the other hand, by Schreier's theorem (\cite[Example~6.3]{Koch}) since $n = \dim J(F)$ then we have $\dim J(K) = |G|(n-1)+1$.  Since $X\oplus Y \subseteq J(K)$ by construction, the result follows.

\section{Computing Galois groups}\label{sec:computing.Galois.groups}

In the previous section we computed the module structure of $J(K)$ and saw that it has a single summand which is not free.  In this section we ask what we can say about the Galois groups of field extensions which correspond to modules in $J(K)$.  We begin with the case of free submodules; note that this result does not rely on the freeness of $G_F(p)$.

\begin{proposition}Suppose  $G$ is a finite $p$-group and that $N \subseteq J(K)$ satisfies $N \simeq \bigoplus_{i=1}^n \mathbb{F}_p[G]$.  Let $L$ be the extension of $K$ corresponding to $N$.  Then $\Gal(L/F) \simeq \left(\mathbb{F}_p[G]\right)^{\oplus n} \rtimes G$.
\end{proposition}

\begin{proof}
We know that $H^i(G,P)=\{0\}$ for any projective $\mathbb{F}_p[G]$-module $P$.  Hence when $i=2$ we see that there is only a single extension of $G$ by $\left(\mathbb{F}_p[G]\right)^{\oplus n}$.
\end{proof}

Now we consider the $X$ summand; whereas free summands only have one possible Galois group, this time the module structure of $X$ is not enough to determine the Galois group.

\begin{lemma}\label{le:extension.problems.for.Omega2}
Suppose that $\Gal(L/K)$ is a finite $p$-group which is isomorphic to $\Omega^{2}(\mathbb{F}_p)$ as an $\mathbb{F}_p[\Gal(K/F)]$-module.  Then $\Gal(L/F)$ is one of two possible groups: one semi direct, and the other not.
\end{lemma}

\begin{proof}
We will show that $H^2(G,\Omega^{2}_{\mathbb{F}_p[G]}(\mathbb{F}_p)) \simeq \mathbb{F}_p$.  By Lemma \ref{le:exact.sequence.for.Heller}, if we let $P = \bigoplus_{i=1}^{d(G)-1} \mathbb{F}_p[G]$ we have the following pair of exact sequences
\begin{align}
\xymatrix{0\ar[r]&\Omega(\mathbb{F}_p) \ar[r] & \mathbb{F}_p[G] \ar[r]^-\varepsilon& \mathbb{F}_p \ar[r] & 0}\label{eq:1st.short.exact.sequence.for.Omega.2.and.1}\\
\xymatrix{0\ar[r]&\Omega^2(\mathbb{F}_p) \ar[r] & P \ar[r] & \Omega(\mathbb{F}_p) \ar[r] & 0.}\label{eq:2nd.short.exact.sequence.for.Omega.2.and.1}
\end{align}
If we consider the long exact sequence that corresponds to exact sequence (\ref{eq:2nd.short.exact.sequence.for.Omega.2.and.1}) we have $$\xymatrix{\displaystyle H^1\left(G,P\right) \ar[r] & H^1\left(G,\Omega(\mathbb{F}_p)\right) \ar[r]^-{\partial} & H^2\left(G,\Omega^2(\mathbb{F}_p)\right) \ar[r] & \displaystyle H^2\left(G,P\right)}.$$  But since higher cohomology groups (i.e., $i \neq 0$) vanish for projective modules, this gives $\partial:H^1(G,\Omega(\mathbb{F}_p)) \stackrel{\sim}{\to} H^2(G,\Omega^2(\mathbb{F}_p))$.  Hence the long exact sequence for cohomology applied to exact sequence (\ref{eq:1st.short.exact.sequence.for.Omega.2.and.1}) gives
$$
\xymatrix{
0 \ar[r]& H^0(G,\Omega(\mathbb{F}_p)) \ar[r] \ar@{=}[d] & H^0(G,\mathbb{F}_p[G]) \ar[r] \ar@{=}[d]& H^0(G,\mathbb{F}_p) \ar[r] & H^1(G,\Omega(\mathbb{F}_p)) \ar[r] & H^1(G,\mathbb{F}_p[G])\ar@{=}[d]\\
&\Omega(\mathbb{F}_p)^G & \mathbb{F}_p[G]^G & &&0.
}
$$ We know that $0 \neq \Omega(\mathbb{F}_p) \subseteq \mathbb{F}_p[G]$, and so $\Omega(\mathbb{F}_p)^G \simeq \mathbb{F}_p[G]^G$ by Lemma \ref{eq:fixed.part.of.free.module} and Corollary \ref{cor:all.ideals.contain.norm}. Hence we have $$\mathbb{F}_p \simeq H^0(G,\mathbb{F}_p) \simeq H^1(G,\Omega(\mathbb{F}_p)) \simeq H^2(G,\Omega^2(\mathbb{F}_p)).$$

Of course $H^2(G,\Omega^2(\mathbb{F}_p))$ parameterizes the extensions of $G$ by $\Omega^2(\mathbb{F}_p)$; the trivial extension (the semi-direct product) corresponds to the trivial class, and the nonzero classes each give rise to isomorphic groups.
\end{proof}

The natural question suggested by Lemma \ref{le:extension.problems.for.Omega2} is to compute the group $\Gal(L/F)$ for the field $L$ that corresponds to the summand of $J(K)$ isomorphic to $\Omega^{-2}(\mathbb{F}_p)$.  We know that it is only one of two possibilities, but is it the split or the non-split extension?  By construction, the maximal elementary $p$-abelian quotient of $\Gal(L/F)$ is the same as the maximal elementary $p$-abelian quotient of $\Gal(K/F)$, and so we see that $\Gal(L/F)$ and $\Gal(K/F)$ have the same number of generators. Hence it must be that $\Gal(L/F)$ is the nonsplit extension.

\section{Module structures of $\Omega^{2}(\mathbb{F}_p)$ and $\Omega^{-2}(\mathbb{F}_p)$ when $\Gal(K/F) \simeq \mathbb{Z}/p\mathbb{Z}\oplus\mathbb{Z}/p\mathbb{Z}$}\label{sec:bicyclic.example}

In this section we give presentations for $\Omega^{2}(\mathbb{F}_p)$ and $\Omega^{-2}(\mathbb{F}_p)$ when the underlying group is $\langle \sigma_1,\sigma_2 \rangle \simeq \mathbb{Z}/p\mathbb{Z}\oplus\mathbb{Z}/p\mathbb{Z}$.  To do this, we take advantage of Lemma \ref{le:exact.sequence.for.Heller}, which gives us a method for realizing $\Omega^2(\mathbb{F}_p)$ through the exact sequence $$\xymatrix{0\ar[r] &\Omega^2(\mathbb{F}_p) \ar[r] &\displaystyle P_1 \ar[r]^-\partial & \mathbb{F}_p[\mathbb{Z}/p\mathbb{Z}\oplus\mathbb{Z}/p\mathbb{Z}] \ar[r]^-\varepsilon & \mathbb{F}_p \ar[r] &0};$$ in this case $P_1$ is generated by $c_1$ and $c_2$, and $\partial(c_i) = \sigma_i-1$.  Since $\Omega^2(\mathbb{F}_p) = \ker(\partial)$, we can quickly see three elements which are conspicuously in this set: $a_0=(\sigma_2-1)c_1-(\sigma_1-1)c_2$, as well as $a_1 = (\sigma_1-1)^{p-1}c_1$ and $a_2 = -(\sigma_2-1)^{p-1}c_2$.  (We have included the minus sign in the definition of $a_2$ to make some relations below simpler to state.)  Moreover, we can surmise the structure of these elements under the action of $\mathbb{F}_p[\mathbb{Z}/p\mathbb{Z}\oplus\mathbb{Z}/p\mathbb{Z}]$: for $\{i,j\}= \{1,2\}$ we have $$(\sigma_i-1)^{p-1}a_0 = (\sigma_j-1)a_i;$$ each $a_i$ generates a free $\mathbb{F}_p[\langle \sigma_j\rangle]$-module; and the module generated by $a_0$ is subject to the relation $(\sigma_1-1)^{p-1}(\sigma_2-1)^{p-1}a_0=0$.  As a consequence, an $\mathbb{F}_p$-spanning set for $\Omega^2(\mathbb{F}_p)$ in this case is given by $\{a_1,a_2\} \cup \{(\sigma_1-1)^k(\sigma_2-1)^\ell a_0: 0 \leq k,\ell \leq p-1 \text{ and }k+\ell < 2p-2\}$.  

In fact, we can prove that these elements form an $\mathbb{F}_p$-basis for $\Omega^2(\mathbb{F}_p)$ using a simple dimension count. From the proof of Theorem \ref{th:main.theorem} we know that $\dim_{\mathbb{F}_p}(\Omega^2(\mathbb{F}_p)) = (2-1)p^2+1 = p^2+1$.  This is precisely the number of elements in our spanning set above, and so they must  be $\mathbb{F}_p$-independent.

To give a visualization for this module, we can represent each of its $\mathbb{F}_p$-generators as a box, and then configure those boxes to depict simple $\mathbb{F}_p[\mathbb{Z}/p\mathbb{Z}\oplus\mathbb{Z}/p\mathbb{Z}]$-relations between them.  Specifically, if the box that represents an element $A$ is directly to the southwest of another box that represents an element $B$, this means that $B^{(\sigma_1-1)} = A$; if instead $B^{(\sigma_2-1)} = A$, then we draw $A$ directly to the southeast of box $B$.  
With these conventions in play, the representation for $\Omega^2(\mathbb{F}_p)$ is then given on the left side of Figure \ref{fig:omegas}.  On the right we recover a depiction of its dual module by simply reversing the actions. 

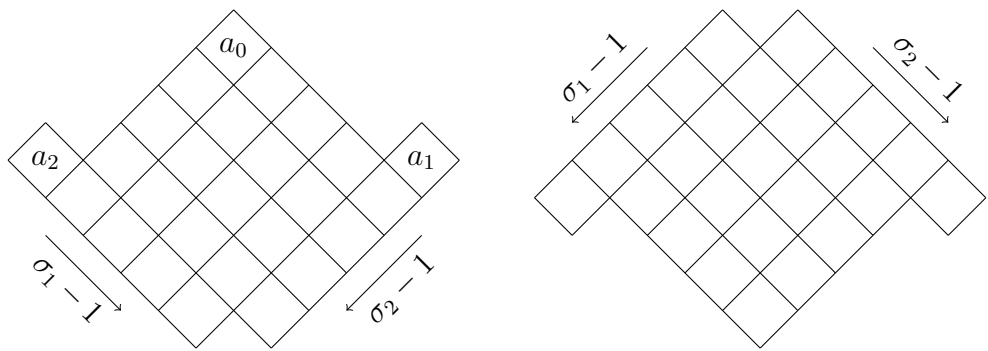
\begin{figure}[h]
\begin{tikzpicture}[x=0.5cm,y=0.5cm]
\draw (0,-1) node {$a_0$};
\draw (-5,-4) node {$a_2$};
\draw (5,-4) node {$a_1$};
\draw[->] (-5,-6) -- node[sloped,below] {$\sigma_1-1$} (-3,-8);
\draw[->] (5,-6) -- node[sloped,below] {$\sigma_2-1$} (3,-8);
\draw (0,0) -- (-5,-5);
\draw (1,-1) -- (-4,-6);
\draw (2,-2) -- (-3,-7);
\draw (3,-3) -- (-2,-8);
\draw (5,-3) -- (-1,-9);
\draw (6,-4) -- (1,-9);
\draw (5,-3) -- (6,-4);
\draw (0,0) -- (5,-5);
\draw (-1,-1) -- (4,-6);
\draw (-2,-2) -- (3,-7);
\draw (-3,-3) -- (2,-8);
\draw (-5,-3) -- (1,-9);
\draw (-6,-4) -- (-1,-9);
\draw (-5,-3) -- (-6,-4);

\draw[->] (11,-1) -- node[sloped,above] {$\sigma_1-1$} (9,-3);
\draw[->] (17,-1) -- node[sloped,above] {$\sigma_2-1$} (19,-3);
\draw (13,0) -- (8,-5);
\draw (15,0) -- (9,-6);
\draw (16,-1) -- (11,-6);
\draw (17,-2) -- (12,-7);
\draw (18,-3) -- (13,-8);
\draw (19,-4) -- (14,-9);
\draw (15,0) -- (20,-5);
\draw (19,-6) -- (20,-5);
\draw (13,0) -- (19,-6);
\draw (12,-1) -- (17,-6);
\draw (11,-2) -- (16,-7);
\draw (10,-3) -- (15,-8);
\draw (9,-4) -- (14,-9);
\draw (8,-5) -- (9,-6);

\end{tikzpicture}
\caption{Graphical depictions of the modules $\Omega^2(\mathbb{F}_p)$ and $\Omega^{-2}(\mathbb{F}_p)$ for $\mathbb{F}_p[\mathbb{Z}/p\mathbb{Z}\oplus \mathbb{Z}/p\mathbb{Z}]$ (when $p=5$).}\label{fig:omegas}
\end{figure}

\end{document}